\documentclass[11pt]{article}
\usepackage{amsfonts}
\usepackage{latexsym}
\usepackage{cite}
\usepackage{amsmath,amsfonts,latexsym,amssymb}
\usepackage[mathscr]{eucal}
\usepackage{cases}
\usepackage{amsthm}

\usepackage[bf,small]{caption2}
\usepackage{float,color}
\usepackage{graphicx}
\usepackage{amsmath}
\usepackage{amssymb}
\usepackage[all]{xy}

\newtheorem{theorem}{theorem}[section]
\newtheorem{thm}[theorem]{Theorem}
\newtheorem{lem}[theorem]{Lemma}

\begin{document}

\title{\textbf{On the groups of periodic links}}
\author{\Large Haimiao Chen
\footnote{Email: \emph{chenhm@math.pku.edu.cn}. The author is supported by NSFC-11401014.} \\
\normalsize \em{Beijing Technology and Business University, Beijing, China}}
\date{}
\maketitle

\begin{abstract}
  It is shown that, if a link $\tilde{L}\subset S^3$ is $p^k$-periodic with $p$ prime and $k\ge 1$, and $L$ is the quotient link, then the groups of $\tilde{L}$ and $L$ can be related by 
  counting homomorphisms to any finite group $\Gamma$ whose order is not divisible by $p$. 
  
\end{abstract}

\section{Introduction}

Let $n$ be a positive integer.
A link $\tilde{L}\subset S^3$ is called {\it periodic} of order $n$ if there exists an orientation preserving homeomorphism $\psi:S^3\to S^3$ such that
$\psi(\tilde{L})=\tilde{L}$, the order of $\psi$ is $n$, and the fixed set of $\psi$ is a circle disjoint from $\tilde{L}$; call $L=\tilde{L}/\psi$ the {\it quotient link}.
Concerning periodicity of links, many results have been obtained, see \cite{BL11,BP17,Ch02,Ch10,Hi81,HLN06,JN13,Po17,QAI15,Sa81}. 

We aim to provide a new approach to study periodic links, via ``Dijkgraaf-Witten invariants" which we introduce right away. 

Given a link $L\subset S^3$, its group is $\pi_1(L)=\pi_1(S^3-N(L))$, where $N(L)$ is a tubular neighborhood of $L$. Suppose the connected components of $L$ are $K_1,\ldots,K_n$. For each $i$, let $\mathfrak{m}_i,\mathfrak{l}_i$ respectively denote the meridian and longitude of $K_i$. The (untwisted) {\it Dijkgraaf-Witeen} invariant of $L$ is defined as follows: for each tuples $\mathbf{x}=(x_1,\ldots,x_n),\mathbf{h}=(h_1,\ldots,h_n)\in\Gamma^n$ with $\mathbf{h}\in{\rm Cen}(\mathbf{x})$, by which we mean $h_i\in{\rm Cen}(x_i)$ (the centralizer) for each $i$,
\begin{align*}
{\rm DW}(L)_{\mathbf{x},\mathbf{h}}=\#\{\phi\in\hom(\pi_1(L),\Gamma)\colon \phi(\mathfrak{m}_i)=x_i, \phi(\mathfrak{l}_i)=h_i, 1\le i\le n\},
\end{align*}
where $\#X$ for a set $X$ is its cardinality.
For $[\mathbf{h}]=([h_1],\ldots,[h_n])$, where each $[h_i]$ stands for the conjugacy class of $h_i$ in ${\rm Cen}(x_i)$, we put
\begin{align*}
{\rm DW}(L)_{\mathbf{x},[\mathbf{h}]}=\#\{\phi\in\hom(\pi_1(L),\Gamma)\colon \phi(\mathfrak{m}_i)=x_i, \phi(\mathfrak{l}_i)\in[h_i],1\le i\le n\}.
\end{align*}

The main result of this note is 
\begin{thm} \label{thm:main}
Suppose $\tilde{L}\subset S^3$ is a $p^k$-periodic link with $p$ prime and $k\ge 1$, and $L$ is the quotient link. Let $n$ be the number of components of $\tilde{L}$ and $L$. Then for any finite group $\Gamma$ with $p\nmid\#\Gamma$, and any $\mathbf{x}\in\Gamma^n$, $\mathbf{h}\in{\rm Cen}(\mathbf{x})$, 
$${\rm DW}(\tilde{L})_{\mathbf{x},[\mathbf{h}^{p^k}]}\equiv {\rm DW}(L)_{\mathbf{x},[\mathbf{h}]}\pmod{p}.$$
\end{thm}

\section{Proof of Theorem \ref{thm:main}}

Let $\mathbb{F}$ denote the algebraic closure of the field with $p$ elements. Recall that there is a {\it Frobenius automorphism} $\mathbb{F}\to\mathbb{F}$ sending $\lambda$ to $\lambda^p$.

\begin{lem} \label{lem:Frob}
For any square matrix $A$ over $\mathbb{F}$, one has ${\rm tr}(A^{p})={\rm tr}(A)^p$.
\end{lem}
\begin{proof}
Let $\lambda_1,\ldots,\lambda_n$ be the eigenvalues of $A$ over $\mathbb{F}$. Since $A$ can be conjugated to an upper-triangular matrix, one has
$${\rm tr}(A^{p})=\lambda_1^{p}+\cdots+\lambda_n^p=(\lambda_1+\cdots+\lambda_n)^{p}={\rm tr}(A)^{p}.$$
\end{proof}

Take a system $S$ of representatives of conjugacy classes of $\Gamma$, and for each $x\in S$, take a system $\mathcal{R}_x$ of representatives of 
isomorphism classes of irreducible representations of ${\rm Cen}(x)$ over $\mathbb{F}$. Let
$$\mathcal{S}=\{(x,\rho)\colon x\in S,\rho\in\mathcal{R}_x\}.$$
Given $r=(x,\rho)\in\Lambda$, denote $x$ by $x_r$ and $\rho$ by $\rho_r$. Let $\rho_r^{\vee}:\Gamma\to{\rm GL}(U(x,\rho))$ denote the representation of $\Gamma$ induced by $\rho$.
For each $h\in\Gamma$, let
$$\chi^\vee_r(h)={\rm Tr}(\rho^\vee_r(h))=\begin{cases} \#[x]\cdot\chi_\rho(h), &h\in{\rm Cen}(x), \\ 0,&h\notin{\rm Cen}(x). \end{cases}$$

Let $B_m$ be the braid group on $m$ strands. We view a braid $\beta\in B_m$ geometrically as follows: Let $C$ be a solid cylinder, with $m$ numbered points on the bottom face and another $m$ numbered points on the top face and connecting $m$ points on the bottom face and those on the top face, then $\beta$ is a tangle connecting the two sets of points; let $\partial^+_i\beta$ (resp. $\partial^-_i\beta$) denote the $i$-th end of $\beta$ on the top (resp. bottom) of $C$. Adopt the convention that $\beta'\beta$ stands for the braid obtained by sticking $\partial^-_i\beta'$ with $\partial^+_i\beta$.

\begin{figure} [h]
  \centering
  \includegraphics[width=0.3\textwidth]{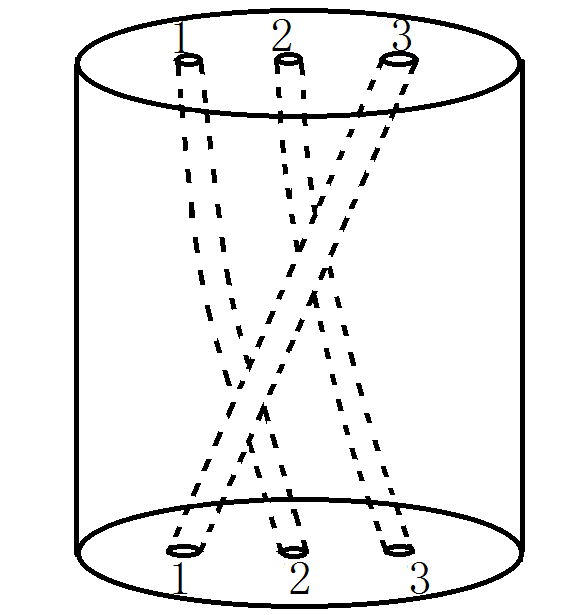}\\
  \caption{The complement of a braid in a solid cylinder}\label{fig:braid}
\end{figure}

For $\beta\in B_m$, we also introduce the following:
\begin{itemize}
  \item let $\sigma^\beta$ denote the permutation on $\{1,\ldots,m\}$ declaring that the $\partial^-_i\beta$ is connected to $\partial^+_{\sigma^{\beta}(i)}\beta$ by a strand $\beta(i)$;
  \item for $V_1,\ldots,V_m$, let
        $P_\beta:V_1\otimes\cdots\otimes V_m\to V_{\sigma^\beta(1)}\otimes\cdots\otimes V_{\sigma^\beta(m)}$
        denote the map sending $v_1\otimes\cdots\otimes v_m$ to $v_{\sigma^{\beta}(1)}\otimes\cdots\otimes v_{\sigma^\beta(m)}$;
  \item let $\beta^\wedge$ denote the {\it closure} of $\beta$ which is a link (in $S^3$);
  \item let $C(\beta)$ denote the complement of a tubular neighborhood of $\beta$ in the cylinder containing it, as illustrate in Figure \ref{fig:braid}.
  \item given $\mathbf{x}=(x_1,\ldots,x_m), \mathbf{a}=(a_1,\ldots,a_m)$, let
        $\mathcal{N}_\beta(\mathbf{x},\mathbf{a})$ denote the number of homomorphisms $\phi:\pi_1(C(\beta))\to\Gamma$ which take $x_i$ on the circle around $\partial^-_i\beta$ and take $a_i$ on the loop containing $\beta(i)$.
\end{itemize}

Given a map $c:\{1,\ldots,,m\}\to\mathcal{S}$, let $\mathbf{x}_c=(x_{c(1)},\ldots,x_{c(m)})$. 
For any $\mathbf{a}=(a_1,\ldots,a_m)\in\Gamma^m$, define
\begin{align}
H(\beta,c;\mathbf{a}):U(c(1))\otimes\cdots\otimes U(c(m))&\to U(c(1))\otimes\cdots\otimes U(c(m)), \\
u_1\otimes\cdots\otimes u_m&\mapsto P_\beta\left(\bigotimes\limits_{i=1}^m\rho^\vee_{c(i)}(a_i)(u_i)\right).
\end{align}
Put
\begin{align}
F(\beta,c)=\frac{1}{(\#\Gamma)^m}\sum\limits_{\mathbf{a}\in\Gamma^m}\mathcal{N}_\beta(\mathbf{x}_c,\mathbf{a})H(\beta,c;\mathbf{a}).
\end{align}

\begin{lem}  \label{lem:composition}
If $c'(\sigma^\beta(i))=c(i)$ for all $i\in\{1,\ldots,m\}$, then
$$F(\beta',c')F(\beta,c)=F(\beta'\beta,c).$$
\end{lem}
\begin{proof}
This is an immediate consequence of von-Kampen Theorem.
\end{proof}

\begin{lem} \label{lem:trace}
Write $\sigma=\sigma^\beta$ as the product of disjoint cycles
$$(i(1,1),\ldots,i(1,e_1))\cdots(i(n,1),\ldots,i(n,e_n)).$$
Suppose $c(i)=c(\sigma(i))$ for each $i$. Let $\mathbf{x}=(x_{i(1,1)},\ldots,x_{i(n,1)})$. 
Then
$${\rm Tr}(F(\beta,c))=\frac{1}{(\#\Gamma)^n}\sum\limits_{\mathbf{h}\in{\rm Cen}(\mathbf{x})}{\rm DW}(\beta^{\wedge})_{\mathbf{x},\mathbf{h}}\cdot\prod\limits_{i=1}^n\rho^\vee_{c(i)}(h_i).$$
\end{lem}
\begin{proof}
Let $\{u_i^j\colon 1\le j\le d_i\}$ be a basis of $U(c(i))$. Write
$$\rho^\vee_{c(i)}(a_i)(u_i^j)=\sum_{\ell=1}^{d_i}\lambda(i)_{\ell,j}u_i^\ell.$$
Then
$$P_\beta\left(\bigotimes\limits_{i=1}^m\rho^\vee_{c(i)}(a_i)(u_i)\right)=\sum\limits_{j_1,\ldots,j_m}\bigotimes\limits_{i=1}^mu_i^{j_i}\prod_{i=1}^m\lambda(i)_{j_{\sigma(i)},j_i}.$$ 
Hence
\begin{align*}
{\rm Tr}(H(\beta,c;\mathbf{a})))
&=\sum\limits_{j_1,\ldots,j_m}\prod\limits_{i=1}^m\lambda(i)_{j_{\sigma(i)},j_i}=\sum\limits_{j_1,\ldots,j_m}\prod\limits_{t=1}^n\prod\limits_{\ell=1}^{e_t}\lambda(i(t,\ell))_{j_{i(t,\ell+1)},j_{i(t,\ell)}} \\
&=\prod\limits_{t=1}^n{\rm Tr}(\lambda(i(t,e_t))\cdots\lambda(i(t,1)))=\prod\limits_{t=1}^n{\rm Tr}(\rho^\vee_{c(t)}(a(t)),
\end{align*}
where $\lambda(i)$ is the matrix whose $(\ell,j)$ entry is $\lambda(i)_{\ell,j}$, and $a(t)$ is the abbreviation for $a_{i(t,e_t)}\cdots a_{i(t,1)}$. 
We have
\begin{align}
{\rm Tr}(F(\beta,c))&=\frac{1}{(\#\Gamma)^m}\sum\limits_{\mathbf{a}\in\Gamma^m}\mathcal{N}_\beta(\mathbf{x}_c,\mathbf{a})\prod\limits_{t=1}^n{\rm Tr}(\rho^\vee_{c(t)}(a(t))) \\
&=\frac{1}{(\#\Gamma)^m}\sum\limits_{\mathbf{h}\in{\rm Cen}(\mathbf{x})}\sum\limits_{\mathbf{a}\in\Gamma^m\colon\atop a(t)=h_t,1\le t\le n}\mathcal{N}_\beta(\mathbf{x}_c,\mathbf{a})\prod\limits_{t=1}^n\chi^\vee_{c(i)}(h_i) \\
&=\frac{1}{(\#\Gamma)^n}\sum\limits_{\mathbf{h}\in{\rm Cen}(\mathbf{x})}{\rm DW}(\beta^{\wedge})_{\mathbf{x},\mathbf{h}}\cdot\prod\limits_{i=1}^n\chi^\vee_{c(i)}(h_i).
\end{align}
\end{proof}

Take $\beta\in B_m$ such that $\tilde{L}=(\beta^{p^k})^\wedge$ and $L=\beta^{\wedge}$. By Lemma \ref{lem:composition}, $F(\beta^{p^k},\tilde{c})=F(\beta,c)^{p^k}$, and then by Lemma \ref{lem:Frob}, Lemma \ref{lem:trace},
\begin{align}
\sum\limits_{\mathbf{h}\in{\rm Cen}(\mathbf{x})}{\rm DW}(\tilde{L})_{\mathbf{x},\mathbf{h}}\cdot\prod\limits_{i=1}^n\chi^\vee_{c(i)}(h_i)
&=\sum\limits_{\mathbf{h}\in{\rm Cen}(\mathbf{x})}{\rm DW}(L)_{\mathbf{x},\mathbf{h}}\cdot\prod\limits_{i=1}^n(\chi^\vee_{c(i)}(h_i))^{p^k} \\
&=\sum\limits_{\mathbf{h}\in{\rm Cen}(\mathbf{x})}{\rm DW}(L)_{\mathbf{x},\mathbf{h}}\cdot\prod\limits_{i=1}^n\chi^\vee_{c(i)}(h_i^{p^k}).
\end{align}
Replacing $\mathbf{h}$ in the left-hand-side by $\mathbf{h}^{p^k}$, we can re-write the equation as
$$\sum\limits_{\mathbf{h}\in{\rm Cen}(\mathbf{x})}{\rm DW}(\tilde{L})_{\mathbf{x},\mathbf{h}^{p^k}}\cdot\prod\limits_{i=1}^n\chi^\vee_{c(i)}(h_i^{p^k})
=\sum\limits_{\mathbf{h}\in{\rm Cen}(\mathbf{x})}{\rm DW}(L)_{\mathbf{x},\mathbf{h}}\cdot\prod\limits_{i=1}^n\chi^\vee_{c(i)}(h_i^{p^k}).$$
By Corollary 5.3.5 of \cite{Be91}, for any group $\Xi$ with $p\nmid\#\Xi$, the number of irreducible characters equals that of conjugacy classes. By Corollary 2 and Corollary 3 on Page 14 of \cite{Se77}, 
any two irreducible characters $\chi,\chi'$ satisfy
$$\frac{1}{\#\Xi}\sum\limits_{g\in\Xi}\chi(g)\chi'(g^{-1})=\delta_{\chi,\chi'},$$
which implies that the irreducible characters of $\Xi$ are linear independent. Applying this to $\Xi={\rm Cen}(x_i)$, we see that for each conjugacy class $[h_i]$, the characteristic function $\delta_{[\mathbf{h}]}$ is a linear combination of the $\chi_\rho$'s for $\rho\in\mathcal{R}_x$. Then the proof is complete.

\end{document}